\documentclass[english]{amsart}
\usepackage[T1]{fontenc}
\usepackage[latin9]{inputenc}
\usepackage{amsthm}
\usepackage{amstext}
\usepackage{amssymb}
\makeatletter
\numberwithin{equation}{section} 
\numberwithin{figure}{section} 
\theoremstyle{plain}
\theoremstyle{plain}
\newtheorem{thm}{Theorem}
  \theoremstyle{plain}
  \newtheorem{prop}{Proposition}

 \usepackage{epsfig} 
 \usepackage{graphics} 
 
\usepackage{graphicx}
\def\C{\mathbb C}
\newcommand{\abs}[1]{| #1 |}

\def\B{\mathbb B}
\def\bcases{\begin{cases}}
\newcommand{\br}[1]{\left(#1\right)}
\newcommand{\bk}[1]{\left[#1\right]}
\def\ecases{\end{cases}}
\newcommand{\cl}{\overline}

\newcommand{\dd}{\delta}

\newcommand{\e}{\epsilon}

\newcommand{\im}{\text{\rm Im}\,}
\newcommand{\inner}[1]{\left\langle #1 \right\rangle}

\newcommand{\norm}[1]{\left\| #1\right\|}

\newcommand{\p}{\partial}
\newcommand{\R}{\mathbb R}

\newcommand{\set}[1]{\left\{ #1\right\}}
\def\sm{\setminus}

\newcommand{\Vol}{\operatorname{Vol}}

\newcommand{\W}{\Omega}

\newcommand{\Z}{\mathbb Z}

\newtheorem{main theorem}{Main Theorem}
\newtheorem{corollary}{Corollary}

\newtheorem{lemma}{Lemma}

\newtheorem{problem 1}{Problem 1}
\newtheorem{problem 2}{Problem 2}
\newtheorem{problem 3}{Problem 3}
\theoremstyle{definition}

\newtheorem{defn}{Definition}
\newtheorem{remark}{Remark}

\newcommand{\bea}{\begin{eqnarray*}}
\newcommand{\eea}{\end{eqnarray*}}

\newcommand{\be}{\begin{equation}}
\newcommand{\ee}{\end{equation}}

\input xy
\xyoption{all}

\makeatother

\usepackage{babel}

\begin{document}

\title[Comparison of Invariant Metrics]{Comparison of Invariant Metrics}

\author{Hyunsuk Kang, Lina Lee, Crystal Zeager}





\thanks{ The third author is
supported by NSF RTG grant DMS-0602191.}


\today
\begin{abstract}
We estimate the boundary behavior of the Kobayashi metric on $\C\sm\set{0,1}$. We also compare the Bergman metric on the ring domain in $\C^{2}$ to the Bergman metric on the ball.
\end{abstract}

\maketitle

\section{Introduction}

The Kobayashi metric and the Carath\'eodory metric are the generalizations of the Poincar\'e metric in higher dimensions and it is known that the Kobayashi metric is the largest and the Carath\'eodory metric is the smallest among such metrics, i.e., metrics that coincide with the Poincar\'e metric on the unit disc in $\C$ and satisfy the non-increasing property under holomorphic mappings. For example, the Sibony metric and the Azukawa metric, whose definitions can be found in section 2, are both examples of such metrics and we always have the following inequality: Carath\'eodory metric $\le$ Sibony metric $\le$ Azukawa metric $\le$ Kobayashi metric. 

On the other hand, the Bergman metric is neither equal to the Poincar\'e metric on the unit disc in $\C$, nor does it satisfy the non-increasing property under holomorphic mappings. Hence the inequality between the Bergman metric and the Carath\'eodory metric or the Kobayashi metric is not obvious. It was K.T. Hahn who proved that the Bergman metric is always greater than or equal to the Carath\'eodory metric in \cite{KH}. Diederich-Forn\ae ss \cite{DF} and Diederich-Forn\ae ss-Herbort \cite{DFH} showed that there is no simple inequality that always holds between the Kobayashi metric and the Bergman metric.

\subsection{Comparison of metrics in $\C$}
The first question we ask is regarding the Kobayashi, Carath\'eodory, Azukawa, and Sibony metrics. As stated above, one can always find a simple inequality between these metrics. Then one may wonder how differently these metrics can behave. That leads to the following question: can one find a domain where one metric does not vanish anywhere but another metric vanishes at some point? In other words, can one find a domain, where the hyperbolicity of these metrics are not equivalent? 

The domain $\C\sm\set{0,1}$ is an example of such a domain. The Kobayashi metric does not vanish everywhere since its covering space is the unit disc and the Kobayashi metric does not vanish on the unit disc. The Carath\'eodory metric, on the other hand, vanishes everywhere, since there does not exist a bounded holomorphic function on $\C\sm\set{0,1}$. The Sibony and Azukawa metrics are also identically zero on $\C\sm\set{0,1}$ because the bounded plurisubharmonic functions on $\C\sm\set{0,1}$ are all constant. 

A higher dimensional example of the same phenomenon is given by $ \W \times \Delta^n $ where  $ \W = \C \sm \set{0,1} $. It is easy to check that this domain is Kobayashi hyperbolic using the known formula for the Kobayashi metric on product domains \cite[p. 106]{Pflug}.  But this domain is not Carath\'eodory, Azukawa, or Sibony hyperbolic because it contains copies of $ \W $.

We then study how the Kobayashi metric behaves as the point approaches one of the punctures and obtained the following theorem: 
\begin{thm} \label{MainThmHyp}
Let $\W=\C\setminus\{0,1\}$, and let $\text{dist}(p,0)=\delta,\xi=1$. Then for $ \delta>0 $ sufficiently small we have 
\begin{equation}
F_{K}^{\W}(p,\xi)\approx\frac{1}{\delta \log{1/\dd}},
\end{equation}
where $F_K^\W(p,\xi)$ denotes the Kobayashi metric on $\W$ at the point $p$ in the direction $\xi$.
\end{thm}

We prove the above result in Section 3 by using the elliptic modular function and calculating its derivatives.

\subsection{Comparison of the Bergman metric on a ring domain in $\C^n$}

If a metric $F$ satisfies the non-increasing property under holomorphic mappings, then it is immediate that $F^A\ge F^B$ if $A\subset B$, where $F^A$ and $F^B$ denote the metric $F$ on $A$ and $B$ respectively, since the metric should not increase under the inclusion mapping. 

We consider a ring domain $\W=\set{r<\norm z<1}\subset\C^n$, $r\in(0,1)$. Since $\W\subset\B^n=\set{\norm z<1}$, we have the following inequality $F^\W\ge F^{\B^n}$, where $F$ denotes any of the Kobayashi, Carath\'eodory, Azukawa and Sibony metrics. In fact, the Sibony, Azukawa, and the Kobayashi metrics blow up near the inner boundary of $\W$ in the normal direction, see \cite{FL}, \cite{Krantz}. 

In section 4, We study how the Bergman metric behaves on a ring domain in $\C^n$. First, we  show that the Bergman metric does not blow up near the inner boundary. In Proposition \ref{627}, we show this in a more general setting: if $K\subset\subset\W\subset\subset\C^n$ and $K$, $\W$ are domains, then $F_B^{\W\sm K}\approx F_B^\W$, where $F_B$ denotes the Bergman metric. We then show the following strict inequality in the tangential direction on a ring domain in $\C^n$:

\begin{thm}\label{602}
Let $\B^n$ be the unit ball in $\C^n$ and  $\Omega$ be the ring domain,  $\Omega = \{z\in\C^{n}:r<|z|<1\}$, $r\in(0,1)$. Let $p\in\W$ and $\xi\in T_p\W$ be such that $\xi\cdot\cl p=0$, i.e., $\xi$ is in the tangential direction to the inner boundary. Then
\begin{equation}
F_{B}^{\W}(p,\xi)\lneq F_{B}^{\mathbb{B}^n}(p,\xi),\quad\forall p\in\W.
\end{equation}
\end{thm}

Note that we have the reverse inequality, $F^\W(p,\xi)\ge F^{\B^n}(p,\xi)$ for the Kobayashi, Carath\'eodory, Azukawa, and the Sibony metric for all $p\in\W$ and for all $\xi\in\C^n$.

We show the strict inequality in the normal direction on a ring domain in $\C^2$ near the inner boundary when the inner radius is small enough, see Proposition \ref{ringnormal}. Hence we obtain the strict inequality in all directions in the ring domains in $\C^2$ in such cases: 

\begin{thm} \label{MainThmBerg}
Let $\B^2$ be the unit ball in $\C^2$ and  $\Omega$ be the ring domain,  $\Omega = \{z\in\C^{2}:r<|z|<1\}$ for $r>0$ small. Let
$p=(r+\epsilon,0)$ for $\e>0$ small  and let $\xi\in\C^{2}$. Then we have 
\begin{equation}
F_{B}^{\W}(p,\xi)\lneq F_{B}^{\mathbb{B}^2}(p,\xi). 
\end{equation}
 \end{thm}


This paper is organized as follows: In Section \ref{Definitions and Background} we give definitions and background for the metrics. We prove Theorem \ref{MainThmHyp} in Section \ref{Kob} and Theorem  \ref{602}, \ref{MainThmBerg} in Section \ref{Ring}.


\section{Definitions and Background} \label{Definitions and Background}
In this section, we give definitions and properties of the metrics
which are used in later sections. For more detailed discussion of
the metrics, see \cite{Pflug}.

\begin{defn} Let $\Omega\subset\C^{n}$ be a domain, $p\in\W$, and $\xi=(\xi_{1},...,\xi_{n})\in\C^{n}$. Let $\Delta$ be the unit disk in $\C$, and let $\B^{n}(p,r)\subset\C^{n}$
be the ball of radius $r$ centered at $p$.

\begin{itemize}
\item  Carath\'eodory pseudometric $F_{C}^{\W}(p,\xi)$ is defined as
 \[ F_{C}^{\W}(p,\xi)=\sup\{|f'(p)\cdot\xi|:f\in\mathcal{O}(\W,\Delta),f(p)=0\},\]
 where $\mathcal{O}(\W,\Delta)$ is the set of holomorphic functions from $\Omega$ to $\Delta$. \\[0.1cm]

 \item {Kobayashi pseudometric} $F_{K}^{\W}(p,\xi)$ is defined as

 \[
F_{K}^{\W}(p,\xi)=\inf\left\{ |\alpha| \ :\ f\in\mathcal{O}(\Delta,\W),f(0)=p,\exists\alpha>0,\alpha f'(0)=\xi\right\}, \]
 where $\mathcal{O}(\Delta, \Omega)$ is the set of holomorphic functions from $\Delta$ to $\Omega$. \\[0.1cm]

\item {Sibony pseudometric}  $F_{S}^{\W}(p,\xi) $ is defined as
$$
F_{S}^{\W}(p,\xi) =\sup\left\{ (\partial\overline{\partial}u(p)(\xi,\overline{\xi}))^{1/2} =\left(\sum_{i,j=1}^{n}\frac{\partial^{2}u(p)}{\partial z_{i}\partial\overline{z_{j}}}\xi_{i} \overline{\xi_{j}}\right)^{1/2}:u\in\mathcal{S}_{\W}(p)\right\},$$
 where
 $\mathcal{S}_{\W}(p)$ is the set of functions $u$ such that $u:\W\to[0,1)$ vanishes at $p$,  $\log u$ is plurisubharmonic, and $u$ is $C^2$ near $p$.

\item {Azukawa pseudometric} $F_{A}^{\W}(p,\xi) $ is defined as
$$
F_{A}^{\W}(p,\xi)  =\sup\left\{ \limsup_{\lambda\searrow0}\frac{1}{|\lambda|}u(p+\lambda\xi): u\in\mathcal{K}_{\W}(p)\right\},   $$

where
$\mathcal{K}_{\W}(p)$ is the set of functions $u$ such that  $ u:\Omega \to[0,1),$ $\log u$ is plurisubharmonic, and there exits $M>0, r>0$  such that  $ \mathbb{B}^n (p,r)\subset\Omega$ and $u(z) \le M\|z-p\| $ for all $z\in \mathbb B^{n}(p,r) .$

\end{itemize}
 \end{defn}

Let $K_\Omega$ be the Bergman kernel of $\Omega.$

\begin{defn} The Bergman metric  $F_{B}^{\W}(p,\xi)$  is defined by 
 \[
F_{B}^{\W}(p,\xi)=\left(\sum_{\nu,\mu=1}^{n}\frac{\partial^{2}}{\partial z_{\nu}\partial\overline{z_{\mu}}}\log K_{\W}(z,z)\xi_{\nu}\overline{\xi_{\mu}}\right)^{1/2}
\]
provided that $ K_\W $ is nonvanishing on $\W$. 
 \end{defn}

Except for the Bergman metric  the other four metrics are non-increasing with respect to  holomprhic mappings, that is, if $\Phi: \Omega_1 \to \Omega_2$ is holomprhic, then
$F^{\Omega_1}(p, \xi) \geq F^{\Omega_2} (\Phi(p), \Phi_*(\xi))$ where $F^{\Omega_i}$ is one of $F_C^{\Omega_i},F_S^{\Omega_i},F_A^{\Omega_i},$ and $ F_K^{\Omega_i}$.
 Moreover,  they  satisfy the following relationship:

\begin{equation} \label{eqn1}
F_{C}^{\W} (p, \xi) \le F_{S}^{\W}  (p, \xi) \le F_{A}^{\W}  (p, \xi) \le F_{K}^\Omega  (p, \xi)
\end{equation}

 for all $p$ and $\xi$. The Bergman metric behaves  differently  from the rest of metrics  in the sense that it does not have non-increasing property, nor does it fit in the comparison \eqref{eqn1}.  Between the  Carath\'eodory and the  Bergman metric the following is known.

\begin{thm} [K. Hahn, \cite{KH} ]
In any complex manifold $\Omega$, the Bergman metric $F_B^\Omega$  is always greater than or equal to the Carath\'eodory differential metric $F_C^\Omega$ if $M$ admits them:
\begin{equation}
F_{C}^\Omega(p, \xi)  \le F_{B}^\Omega (p, \xi).
\end{equation}
\end{thm}

However, Kobayashi  and Bergman metrics do not have any such relation and  they are in fact incomparable as discussed in the introduction. 

The boundary behavior of the Sibony metric on the ring domain in $\C^{2}$ near the inner boundary
is a special case of domains  studied in \cite{FL}.

\begin{thm} [Forn\ae ss-Lee, \cite{FL}]
Let $\Omega =\{  \frac14 < |z_1|^2 + |z_2|^2 < 1\} \subset \C^2$ and $P_\delta = ( 1/2 + \delta, 0)$, and $\xi=(1,0)$. Then we have

\[
F_{S}^{\W}(p,\xi)\approx\frac{1}{\delta^{1/2}}\]
for $\delta>0$ small enough.
\end{thm}

In particular since the Sibony metric blows up towards
the inner boundary in the normal direction, so does Azukawa metric by \eqref{eqn1}.   In section 4 we show that Bergman is bounded in the inner boundary, hence not comparable with either the Sibony or the Azukawa metric.

 For the convenience to readers, we provide the formula for Bergman kernel and Bergman metric on the unit ball: 
\[ K_{\mathbb B^n}(z, \zeta) = \frac{n!}{\pi^n} \frac{1}{ (1-\sum z_j\bar \zeta_j)^{n+1}} 
\] 
and
\begin{equation} \label{eq123} 
 F_B^{\mathbb B^n}(p; \xi) = \sqrt{n+1} \biggl(  \frac{|\inner{p,\xi}|^2  }{(1-|p|^2 )^2} + \frac{|\xi|^2}{1-|p|^2} \biggr)^{1/2}
\end{equation} 
 where $z=(z_1, \ldots, z_n),$ and $\zeta = (\zeta_1, \ldots, \zeta_n).$


\section{Boundary behavior of the Kobayashi metric on $ \C \sm \set{0,1} $} \label{Kob}

One technique for studying the Kobayashi metric of a domain is to study instead the Kobayashi metric of its covering space.  In general, if $ \pi:\tilde\Omega \to \Omega $ is the covering map it is known that $ F_K^{\tilde{\W}} = \pi^* F_K^{\W} $ \cite[p. 91]{Kobayashi}. In the case that the covering space is a half-plane $ \mathbb{H} $ we have the following equation. We provide a proof for the convenience of the reader.


\begin{lemma} \label{CovSp}
Let $ \Omega $ be a connected domain in $\C$  whose universal covering is $\mathbb{H}$. Let $q\in \mathbb{H}$ and let $m: \mathbb{H} \to \Delta $ be the biholomorphism  such that $m(q) = 0$. 
Let $ p \in \Omega $, $ \xi \in \C^n $ and $\pi: \mathbb{H} \to \Omega$, with $\pi(q) = p$. Then 

\[
F_K^\W (p,\xi) = \frac{ |m'(q)|}{|\pi'(q)|}\|\xi\| . 
\] 
\end{lemma}

\begin{proof}
Let $ f $ be a candidate for the Kobayashi metric at the point $p$, that is, $f(0) = p$ and $f'(0) $ is the multiple of $\xi$.  Since the unit disc is simply connected  there exists the unique lifting  $ \tilde{f} : \Delta \to \mathbb{H} $ such that $ \tilde f(0) = q$ as in the following diagram. 
\[
\xymatrix{
& \Delta
&
&
& 0 \\
& \mathbb{H} \ar[u]_{m(z)} \ar[d]^{\pi}
&
&
& q \ar @{|->}_m [u] \ar @{|->}^\pi [d] \\
\Delta \ar[r]_{f} \ar @{-->} [ur]^{\tilde{f}}
& \W
&
& 0 \ar @{|->}_f [r] \ar @{|->}^{\tilde{f}} [ur]
& p
}
\]
Let $\pi^{-1}$ be the local inverse in a neighborhood of  $p$.  Since $ m \circ \tilde{f} (0) = 0$ and $ \tilde{f} = \pi^{-1} \circ f $   the Schwarz Lemma gives 

\[
|m'(q) \cdot (\pi^{-1})'(p) \cdot f'(0)| \le 1,
\]
implying
\[
\frac{1}{|f'(0)|} \ge \frac{|m'(q)|}{|\pi'(q)|}. 
\]

This estimate holds for any candidate function and the map $ \pi \circ m^{-1} $ is itself a candidate for the Kobayashi metric. The claim follows.

\end{proof}

We will use Lemma \ref{CovSp} to estimate boundary behavior of the Kobayashi metric on the domains $ \Delta \setminus \set{0} $ and $ \C \sm \set{0,1} $.


\begin{prop} \label{WPrime}
Let  $p$ be the point in $ \Delta\setminus \set{0}$ such that  $ \text{dist}(p,0)= \delta $ and let $\xi = 1$. 
For  every $ \delta>0 $  we have 

\begin{equation*}
F_K^{\Delta \sm \set{0}} (p,\xi) = \frac{1}{2\delta \log1/\delta}.
\end{equation*} 

\end{prop} 

\begin{proof} 
The map $z\mapsto z e^{i\theta}$ is a self map of $\Delta \setminus \{0\}$, hence it is enough to show the proposition in the case that $p=\delta$. 
Let $ \mathbb{H}_{\text{left}} $ denote the left half plane $ \set{{\rm Re}(z) < 0} \subset \C $.  The covering map is given by $\pi: \mathbb{H}_{\text{left}} \to \Delta\setminus\{0\}, z \mapsto e^z$.  Let $q = \log \delta$. 
Then, $m: \mathbb{H}_{\text{left}} \to \Delta$ is 
\[
m(z) := \frac{z-\log\delta}{z+\log\delta} \text{  \ with \  } m'(q) = \frac{1}{2\log \delta}.  
\]
By  Lemma \ref{CovSp} we have 
\[
F_K^{\Delta \sm \set{0}} (\delta,1) =  \frac{1}{2\delta | \log\delta|}=  \frac{1}{2\delta \log(1/\delta)}. 
\]

\end{proof}


In order to use Lemma \ref{CovSp} to obtain a boundary estimate for $ \W = \C \sm \set{0,1} $ we consider the elliptic modular function as the covering map from the half-plane to $ \W $ and  estimate its derivative. Let $\lambda(\tau) $ be the elliptic modular function. We include some properties and estimates related to $\lambda$ to be used for our purpose.  For a more complete reference on $\lambda$,  see \cite[Ch 7. Section 3.4]{Ahlfors}.

It is known that the elliptic modular function can be written as follows: 

\begin{equation} \label{EMFrac}
\lambda(\tau) = \frac{\sum_{n=-\infty}^\infty \left[ \frac{1}{cos^2(\pi(n-\frac{1}{2})\tau)} - \frac{1}{sin^2(\pi(n-\frac{1}{2})\tau)} \right]}{\sum_{n=-\infty}^\infty \left[ \frac{1}{cos^2(\pi n \tau)} - \frac{1}{sin^2(\pi(n-\frac{1}{2})\tau)} \right]} =: \frac{N(\tau)}{D(\tau)} .
\end{equation}

We first restrict $\lambda$ on a strip shaded in the figure.  The figure shows how  $ \lambda $ maps the  region  and its boundary to the upper half plane.  By the  Schwarz reflections  we can see images of $ \lambda $  will then cover the complex plane except for $\{0, 1\}$ since they are the  images of boundary points $\{0, 1\}$ of the upper half plane. 
 
 \begin{figure}\label{fig1}
  \begin{center}
\includegraphics[scale=.6]{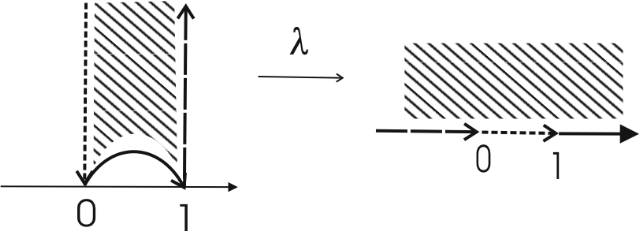}
  \end{center}
    \caption{}
\end{figure}


 We also notice that $ \lambda $ maps the infinity to the origin.   To  estimate the Kobayashi metric on $ \C \sm \set{0,1} $ near the origin at the point $p$,  we will consider the preimage $ q = \lambda^{-1}(p) $ in the upper half-plane  with large positive  imaginary part. 

The following convergence results are known (see \cite[p. 280]{Ahlfors}).  Uniformly with respect to $ {\rm Re}(\tau) $, as $ {\rm Im}(\tau) \to \infty $

\begin{eqnarray}  
D(\tau) \to \pi^2,  \label{denominator} \\
 \lambda(\tau) e^{-i\pi\tau} \to 16.   \label{EMEst}
\end{eqnarray}

\begin{lemma} \label{BdDen}
As $\im(\tau) \to \infty$, the derivatives of  $N(\tau)$  is uniformly bounded by a constant: 
\begin{eqnarray} 
&& \left| \frac{d}{d\tau} D(\tau) \right|  <C
\end{eqnarray} 
for some $C>0$ and  the derivatives of  $D(\tau)$  satisfies the following estimate: 

\begin{eqnarray} 
&& \left| \frac{d}{d\tau} N(\tau) \right| \lesssim |e^{i\pi\tau}| = e^{-\pi \im(\tau)}.
\end{eqnarray} 
\end{lemma}

\begin{proof} [Proof of Lemma \ref{BdDen}]
For $z = x+ iy$ we have 
$$ \sin(z) = \frac{1}{2i} ( e^{i(x+iy)} - e^{-i(x+iy)} ) = \frac{1}{2i} ( e^{ix} e^{-y} - e^{-ix} e^y).$$ 
Hence, for $|y|>\frac12 \ln2$,  we have 
\begin{eqnarray} 
\frac14 e^{|y|} < |\sin z| < e^{|y|}. \label{eqn44}
\end{eqnarray} 
Similarly, we have 
\begin{eqnarray} 
\frac14 e^{|y|} < |\cos z| < e^{|y|}. \label{eqn45}
\end{eqnarray} 

Using \eqref{eqn44} and \eqref{eqn45} the derivatives of  cosine  terms  of $D(\tau)$   except for the term with  $n=0$  can be estimated as 
\begin{eqnarray} \label{309} 
\biggl |  \frac{d}{d\tau}     \frac{1}{\cos^2(\pi n \tau)} \biggr|  
&=&\biggl| \frac{2\pi n\sin(\pi n \tau)}{\cos^3(\pi n \tau)}  \biggr|  
\lesssim \frac{ |n|}{ e^{2 \pi |n| \im(\tau)}} 
\end{eqnarray} 
as $\im(\tau) \to \infty$.

Similarly, we have 

\begin{equation} \label{310}
\biggl |  \frac{d}{d\tau}    \frac{1}{\sin^2(\pi(n-\frac{1}{2})\tau)}   \biggr| 
= \biggl|   \frac{2 \pi(n-\frac{1}{2})  \cos(\pi(n-\frac{1}{2})\tau)}{\sin^3(\pi(n-\frac{1}{2})\tau)}   \biggr|  
  \lesssim \frac{|n- 1/2|}{ e^{2 \pi |n-1/2| \im(\tau)}}, 
\end{equation} 
and 
\begin{equation} \label{311}
  \biggl |  \frac{d}{d\tau}    \frac{1}{\cos^2(\pi(n-\frac{1}{2})\tau)}   \biggr| 
 =\biggl|   \frac{2 \pi(n-\frac{1}{2})  \sin(\pi(n-\frac{1}{2})\tau)}{\cos^3(\pi(n-\frac{1}{2})\tau)}   \biggr| 
  \lesssim \frac{ |n-1/2|}{ e^{2 \pi |n-1/2| \im(\tau)}} 
\end{equation} 
for all $n\in \Z$ as $\im(\tau) \to \infty$. 

Hence, the derivative of $D(\tau)$   we have

\begin{eqnarray*} 
   \left| \frac{d}{d\tau} D(\tau) \right| 
&\leq &\sum  \left|   \frac{d}{d\tau}   \frac{1}{\cos^2(\pi n \tau)} \right|+ \sum  \biggl|  \frac{d}{d\tau}    \frac{1}{\sin^2(\pi(n-\frac{1}{2})\tau)} \biggr| . 
\end{eqnarray*} 

By \eqref{309} and \eqref{310}, it is uniformly bounded above.  Similarly  the derivative of $N(\tau)$ is 
\begin{eqnarray*}
\left| \frac{d}{d\tau} N(\tau) \right|   &\lesssim&   \sum \biggl|   \frac{d}{d\tau}  \frac{1}{\sin^2(\pi(n-\frac{1}{2})\tau)}   \biggr|  + \sum 
  \biggl |  \frac{d}{d\tau}    \frac{1}{\cos^2(\pi(n-\frac{1}{2})\tau)}   \biggr|  \\
  &\lesssim& \frac{1}{e^{\pi \im(\tau)}}.
  \end{eqnarray*} 

\end{proof}


\begin{proof} [\bf Proof of Theorem \ref{MainThmHyp}]
Let $\Omega = \C\setminus\{0, 1\}$ and let $ \mathbb{H}_{upper} $ denote the upper half-plane $ \set{ {\rm Im}(z) > 0 } \subset \C $.     Let $p$ be a point close to the origin and let    $\lambda : \C\to  \C\setminus\{0, 1\}$  be the  elliptic modular function as the covering map of $\Omega$. From Figure~3.1  we see  that  an inverse image  of $p$  is  a point $q$  of the form $r + iM$ where $M>0, M \to \infty$ as $p \to 0$.   From Schwarz reflection argument   if we allow $r$ to be in $[0,2]$  then $\lambda (q)$ would be approaching  the origin in every direction. 
We estimate the Kobayashi metric at the point $p$ by applying Lemma \ref{CovSp}.   

 The M\"obious transformation  $m: \mathbb{H}_{upper} \to \Delta$ sending $q= r+ iM $ to the origin 
 is given by 
\[
m(z) =  \frac{z - q}{z- \bar q} = \frac{z-(r+iM)}{z-(r-iM)}.
\]
Lemma \ref{CovSp} gives that

\begin{equation} \label{OmegaEst}
F_K^\W (p,\xi) =  \frac{ |m'(q)|}{|\lambda'(q)|} 
= \frac{1}{2M|\lambda'(r+iM)|}.
\end{equation}

Using the notation we used in \eqref{EMFrac} the derivative of $\lambda$  is  
\begin{equation*} 
\lambda '(\tau) = \frac{ N'(\tau)}{D(\tau)}  - \lambda(\tau)  \frac{D'(\tau)}{D(\tau)}. 
\end{equation*} 

From  \eqref{denominator}, \eqref{EMEst}  and Lemma~\ref{BdDen}  we have 

\begin{equation*} 
|\lambda '(\tau)|  \approx   |N'(\tau) - \lambda(\tau) D'(\tau) |  \  \lesssim | \  e^{i\pi\tau} |
\end{equation*} 

as  $ {\rm Im}(\tau) \to \infty $.  Hence combining with  \eqref{OmegaEst} we have 
\begin{equation} \label{OmegaLowerBd}
F_K^\W (p,\xi) \gtrsim \frac{1}{2Me^{-\pi M}}.
\end{equation}

The equation  \eqref{EMEst} gives $ \delta = \text{dist}(p,0) \approx e^{- \pi M}$, that is, $ M \approx \log (1/\delta)$. It follows from \eqref{OmegaLowerBd}  we have lower bound estimate : 
\begin{equation*}
F_K^\W (p,\xi) \gtrsim \frac{1}{\delta \log(1/\delta) }.
\end{equation*}

On the other hand $\Delta\setminus\{0\}$  is a subset of $\C\setminus\{0, 1\}$. By non-increasing property of Kobayashi metric we have 
$$  F_K^\W (p,\xi) \leq F_K^{\Delta \setminus \{0\}} (p,\xi) = \frac{1}{2\delta \log1/\delta}.  $$  
This completes the proof. 
\end{proof}

\begin{remark}
The Kobayashi metric has the same boundary beahvior near $ 1 $ as it does near $ 0 $ on the domain $ \C \sm \set{0,1} $. One can see this by noting that the map $ z \mapsto 1-z $ is a biholomorphism that exchanges the points $ 0 $ and $ 1 $ and leaves length of the tangent vector unchanged.
\end{remark}

\begin{corollary}
Let $\W\subset\C$ be a domain with discrete punctures, i.e., $\W=U\sm J$, where $U$ is a domain in $\C$ and $J$ is a union of discrete points in $U$ with at least two points and  no limit point. As $p\in\W$  approaches  a point $p_j\in J$  the Kobayashi metric on $\Omega$ satisfies the estimate 
 \[
F_K^{\W} (p,\xi) \approx \frac{1}{\dd \log (1/\dd)}
\] 
for $ \|\xi\| = 1 $. 

\end{corollary}

\begin{proof}
Let $ p_{j'} \in J $ with $ j' \neq j $. Then since the points in $J$ are discrete there exists a disk $ \Delta(p_j,r) $ of radius $r$ centered at $p_j$ such that $ \Delta(p_j,r) \sm \set{p_j} \subset \W  $. We have the following inclusions 
\[
\Delta(p_j,r) \sm \set{p_j} \subset \W  \subset \C \sm \set{p_j,p_{j'}}
\]
giving us  corresponding inequalities of the metrics:
\[
F_K^{\Delta(j,r) \sm \set{j}} \ge F_K^{\W} \ge F_K^{\C \sm \set{j,j'}}.
\]
Proposition~\ref{WPrime} and Theorem \ref{MainThmHyp} imply the Corollary. 
\end{proof}

%
%

\section{Estimation of the Bergman Metric Near the Inner Boundary of the Ring
Domain} \label{Ring}

We show first that if $\Omega$ is a bounded  domain in $\C^n$ and $K$ is a relatively compact subset of $\Omega$ then the Bergman metric on $\W\sm K$ is comparable to the Bergman metric on $\W$. Hence the metric does not blow up on $\W\sm K$ near $\partial K$.   We denote the space  of square integrable functions holomorphic  in $\Omega$ by $L^2_h(\Omega)$. 
\begin{prop}\label{627}
Let $\W\subset\subset\C^n, n\geq 2$ be a bounded domain and $K\subset\subset\W$ be a domain in $\C^n$. Then there exists a constant $d$, $0<d<1$ depending on $K$ such that 
 
 $$
\sqrt{1-d^2}F_B^\W(z,\xi)\le F_B^{\W\sm K}(z,\xi)\le \frac{1}{\sqrt{1-d^2}} F_B^\W(z,\xi).
$$

\end{prop}
\begin{proof}
We use the following property for the Bergman metric: 
\begin{equation} \label{622}
\left(F_B^\W(z,\xi)\right)^2=\frac{b_\W^2(z,\xi)}{K_\W(z,z)}
\end{equation} 
where
\begin{eqnarray*}
&&b_\W^2(z,\xi)=\sup\set{\abs{\inner{\p f,\xi}}^2:f\in L_h^2(\W), f(z)=0, \norm f_{L^2(\W)}=1}\\
&& \text{and }\\
&&K_\W(z,z)=\sup\set{|g(z)|^2:\norm g_{L^2(\W)}=1, g\in L_h^2(\W)}.
\end{eqnarray*} 

Every holomorphic function $f$ on $\Omega\setminus K$ holomorphically extends to $\Omega$. 
We use the same $f$ for extension.  

We shall show that there exists $d$ such that 
$$
\norm f_{L^2(K)}<d <1 \quad \text{ for all\ }  f\in L^2_h(\Omega) \text{ \ with \ } \norm f_{L^2(\W)}=1. 
$$

If not, there exists a sequence $f_n$ with $ \norm f_{L^2(K)}\geq 1-1/n$ for all $n \in \mathbb N$. Since $L^2$-norm of $f_n$ is uniformly bounded, $|f(z)|$ is also uniformly bounded for all $z\in K$.   Hence there exists a convergent subsequence, call it again as $f_n$,  whose limit function $F$ is also  in $ L_h^2(\W)$. 
We have 
$$ \|F\|_{L^2(K)} = \lim  \|f_n\|_{L^2(K)}  = 1 \text{\  whereas \ } \|F\|_{L^2(\Omega)} \leq 1$$ 
 implying $F = 0$ almost everywhere on   $\W\sm K$,  hence $F\equiv 0$. This contradicts  $  \|F\|_{L^2(K)} = 1$. \\

We have for $f\in L_h^2(\W) \setminus \{0\}$
\begin{equation*} 
(1-d^2)\norm f_{L^2(\W)}^2<\norm f_{L^2(\W\sm K)}^2<\norm f_{L^2(\W)}^2. 
\end{equation*} 
Equivalently,   
\begin{equation} \label{623} 
\frac{1}{\norm f_{L^2(\W)}^2} <  \frac{1}{\norm f_{L^2(\W\sm K)}^2}  < \frac{1}{(1-d^2)\norm f_{L^2(\W)}^2} . 
\end{equation} 
From \eqref{622} and \eqref{623}  we have proposition. 
\end{proof}

%
%

We now  consider the case that $\Omega$ is the unit ball in $\C^2$ and $K$ is a ball of smaller radius centered at the origin.  In this special case we can obtain a specific  constant {\it d} using the orthonormal basis. 
We let  $ \B=\set{z\in\C^2:|z_1|^2+|z_2|^2<1}$ and $\W_r=\set{z\in\C^{2}:r^{2}<|z_1|^{2}+|z_2|^2<1}$. 
Let $ \{ a_{jk}z_{1}^jz_2^k \}_{ j\geq 0, k\geq 0} $ be an orthonormal basis of $L_{h}^{2}(\B)$.  
Then we can easily check that    $ \{\frac{1}{r^{j+k+2}} a_{jk}z_{1}^jz_2^k  \}_{ j\geq 0, k\geq 0} $ is  an orthonormal basis of $L_{h}^{2}(r\B)$.  If  $f(z) = \sum c_{jk} z_1^j z_2^k$  then we have 
$$ \|f \|^2_{L^2( r\mathbb B)} = \sum \frac{|c_{jk}|^2}{ |a_{jk}|^2}  r^{2j+2k+ 4} < r^4  \sum \frac{|c_{jk}|^2}{ |a_{jk}|^2}   = r^4  \|f \|^2_{L^2( \mathbb B)}.$$
Hence, we can take $d = r^2$ and apply Proposition~\ref{627} to have 
\begin{equation}\label{786}  
\sqrt{1-r^4}F_B^{\mathbb B} (z,\xi) \le F_B^{\W_r}(z,\xi)\le \frac{1}{\sqrt{1-r^4}} F_B^{\mathbb B}(z,\xi).
\end{equation}  
One can see that if $\Omega$ is the unit ball in $\C^n$, then we can take $d=r^n$ and have similar inequalities to \eqref{786}. 
In the rest of this section we focus on improving the upper bound.   We prove that  the Bergman metric on $\W_{r}$ is {\it strictly smaller}  than the Bergman metric on $\B$ near the inner boundary when the inner radius  $r$ is small enough.    
To do the comparison we  express the Bergman kernel $K_{\Omega_r}(z,z)$  in terms of the renormalized basis of $L^2_h(\mathbb B)$ and compute the Levi form of $\log K_{\Omega_r}$ in normal and tangential directions.  Let $ \{ a_{jk}'z_{1}^{j}z_2^k \}_{ j\geq 0, k\geq 0} $ be an orthonormal basis of $L_{h}^{2}(\W_r)$.  The coefficients $a'_{jk}$ is a multiple of $a_{jk}$:
\begin{equation} \label{eq411}
 a'_{jk} = a_{jk} \frac{1}{\sqrt{1-r^{2(j+k)+4}}}.
 \end{equation} 
 
This is because 
\begin{eqnarray*} 
\int_{\Omega_r}  |z_{1}|^{2j} |z_2|^{2k} d\Vol  &= &   \int_{\mathbb B} |z_{1}|^{2j} |z_2|^{2k} d\Vol    - \int_{\|z\|<r}  |z_{1}|^{2j} |z_2|^{2k} d\Vol    \\[0.1cm]
& =&  \int_{\mathbb B} |z_{1}|^{2j} |z_2|^{2k} d\Vol    -r^{2(j+k)+4}  \int_{\mathbb B}  |z_{1}|^{2j} |z_2|^{2k} d\Vol. 
\end{eqnarray*} 
The Bergman kernel $K_{\Omega_r}(z,z)$ is   $ \sum | a'_{jk}|^2 |z_{1}|^{2j}|z_{2}|^{2k}$.  
For later  convenience we compute the derivatives of $K_{\Omega_r}$: 

\begin{equation} \label{KernelDerivative}
\frac{\p^{2}\log K_{\W}(z,z)}{\p z_{i}\p\cl z_{j}}=\frac{1}{K^{2}_{\Omega}}\Bigg(K_{\W}(z,z)\frac{\p^{2}K_{\W}(z,z)}{\p z_{i}\p\cl z_{j}}-\frac{\p K_{\W}(z,z)}{\p z_{i}}\frac{\p K_{\W}(z,z)}{\p\cl z_{j}}\Bigg).
\end{equation}\\

We  now prove Theorem~\ref{602} that the Bergman metric on the ring domain is {\it strictly smaller}  than the Bergman metric on the unit ball in the tangential direction. \\
\begin{proof}[{\bf Proof of Theorem~\ref{602}}  ] 
Since the Bergman metric is invariant under biholomorphic mappings,
we may assume $z_{0}=(x,0)$, $x=r+\epsilon, \epsilon>0$ and $T=(0,1)$. Let $c_{j}'=|a_{j1}'|^{2}$
and $c_{j}=|a_{j1}|^{2}$. Then using equation \eqref{KernelDerivative} we have 

\[
(F_{B}^{\W}(z_{0},T))^{2} =\frac{\p^{2}\log K_{\W}(z,z)}{\p z_{2}\p\cl z_{2}}\Bigg|_{z=z_{0}} 
=\frac{\sum_{j=0}^{\infty}c_{j}'x^{2j}}{\sum_{j=0}^{\infty}b_{j}'x^{2j}}.
\]

Similarly  we have 

\[
(F_{B}^{\B}(z_{0},T))^{2}= \frac{\sum_{j=0}^{\infty}c_{j}x^{2j}}{\sum_{j=0}^{\infty}b_{j}x^{2j}}.
\]

Again from the equation \eqref{eq411}, $a'_{jk} = a_{jk}/\sqrt{1-r^{2(j+k)+4}}$  with $k=1$  we have 

\begin{equation} \label{BetaJ} 
c_{j}'=\beta_{j}c_{j}   \text{\  where \ } \beta_{j}=\frac{1}{1-r^{2j+6}}.
\end{equation}

The $c_{j}$'s can be written explicitly : 
\[
c_{j}=\frac{(j+1)(j+2)(j+3)}{\pi^{2}}.
\]
 See \cite[p. 172] {Pflug}.  Recall that the formulas for $ b_j , \gamma_j, $ and $ \beta_j $ are given by equations (\ref{BJ}) and (\ref{BetaJ}) and that $ b_j' = b_j \gamma_j $. We want to show the following:

\[
F_{\W}(z_{0},T)=\frac{\sum_{j=0}^{\infty}\beta_{j}c_{j}x^{2j}} {\sum_{j=0}^{\infty}\gamma_{j}b_{j}x^{2j}}<\frac{\sum_{j=0}^{\infty}c_{j}x^{2j}} {\sum_{j=0}^{\infty}b_{j}x^{2j}}=F_{\B}(z_{0},T). 
\]

We cross multiply to have  
\begin{eqnarray*}
&&\sum\beta_{j}c_{j}x^{2j}\sum b_{k}x^{2k}-\sum c_{j}x^{2j}\sum\gamma_{k}b_{k}x^{2k}\\
&=&\sum_{j}(\beta_{j}-\gamma_{j})b_{j}c_{j}x^{4j}+\sum_{j>k\ge0} (\beta_{j}c_{j}b_{k}+\beta_{k}c_{k}b_{j}-c_{j}\gamma_{k}b_{k}-c_{k} \gamma_{j}b_{j})x^{2(j+k)}\\
&=&\sum_{j}(\beta_{j}-\gamma_{j})b_{j}c_{j}x^{4j}+ \sum_{j>k\ge0}\bk{c_{j}b_{k}(\beta_{j}-\gamma_{k})+c_{k}b_{j} (\beta_{k}-\gamma_{j})}x^{2(j+k)}
\end{eqnarray*}

 Note that the terms in the first summation are all negative since
$\beta_{j}<\gamma_{j}$. Let 

\[
A_{jk}=c_{j}b_{k}(\beta_{j}-\gamma_{k}) +c_{k}b_{j}(\beta_{k}-\gamma_{j}),
\]

Then for every $j, k$, 

\begin{eqnarray*}
A_{jk}&=&\frac{(j+2)!(k+2)!}{\pi^{4}j!k!} \br{(j+3)(\gamma_{j}-\beta_{k})+(k+3)(\gamma_{k}-\beta_{j})}\\
&=&\frac{(j+2)!(k+2)!}{\pi^{4}j!k!} \br{(j+3)\Bigg(\frac{r^{2j+6}-r^{2k+4}}{(1-r^{2j+6})(1-r^{2k+4})}}\\
&&+(k+3)\bk{\frac{r^{2k+6}-r^{2j+4}}{(1-r^{2k+6})(1-r^{2j+4})}}\Bigg)\\
&=&\frac{(j+2)!(k+3)!r^{2k+4}}{\pi^{4}j!k!} \Bigg[-\frac{j+3}{k+3}\br{\frac{1-r^{2(j-k)+2}}{(1-r^{2j+6})(1-r^{2k+4})}}\\
&&+r^{2}\br{\frac{1-r^{2(j-k-1)}}{(1-r^{2k+6})(1-r^{2j+4})}}\Bigg]\\
&=&\frac{(j+2)!(k+3)!r^{2k+4}}{\pi^{4}j!k!}\bk{B_{jk}}
\end{eqnarray*}

 The term $B_{jk}<0$ if 

\be
B:=r^{2}\frac{1-r^{2(j-k-1)}}{1-r^{2(j-k)+2}} \frac{(1-r^{2j+6})(1-r^{2k+4})}{(1-r^{2k+6})(1-r^{2j+4})}<\frac{j+3}{k+3}, \quad\forall j>k\ge0.\label{510}
\ee

Since $j>k$, it is sufficient to show that $B<1$.

If $j-k=1$, then $B=0$. Hence \eqref{510} is satisfied. If $j-k>1$, then we have the following:
$$
B=r^2\frac{h_{2(j-k)-3}}{h_{2(j-k)+1}}\frac{h_{2j+5}}{h_{2k+5}}\frac{h_{2k+3}}{h_{2j+3}},
$$
where $h_n=1+r+\cdots+r^{n}$. Since $h_n/h_m<1$ if $n<m$, it is enough to show that
\be\label{533}
\frac{h_{2j+5}}{h_{2k+5}}\frac{h_{2k+3}}{h_{2j+3}}<1.
\ee
Since $a/b<(a+\e)/(b+\e)$, if $0<a<b$ and $\e>0$, we get the following inequality:
$$
\frac{h_{2k+3}}{h_{2j+3}}<\frac{h_{2k+3}+r^{2j+4}+r^{2j+5}}{h_{2j+5}}<\frac{h_{2k+5}}{h_{2j+5}},
$$
where the second inequality follows from $j>k$ and $r\in(0,1)$. Hence \eqref{533} is proved. 

Therefore we have $B_{jk}<0$ for
all $j>k\ge0$ and $F_K^{\W_r}(z_{0},T)\lneq F_{B}^\B(z_{0},T)$.

\end{proof}
Theorem~\ref{602} can be easily generalized  to higher dimensions and also to a polydisc minus polydisc. We do not know at this point whether the same estimate holds on other rings of Reinhardt domains.

In the normal direction we have the estimate for the ring domain in $\C^2$ when the inner boundary is close to zero. Our computation is restricted to the case of the ring domain in $\C^2$. 
\begin{prop} \label{ringnormal}
Let $\mathbb B^2$ be the unit ball in $\C^2$ and $\Omega $ be the ring domain, $\Omega = \{ z\in \C^2 \ : \ r< |z| < 1\}, r\in (0, 1)$. 
Let $z_0\in \Omega$.  For sufficiently small $r>0$  we have 
\[
F_B^{\W_{r}}(z_{0},N)\lneq F_B^{\B^2}(z_{0},N),\quad N=z_0/|z_0|,
\]
for $z_{0}$ close to the inner boundary of $\W_{r}$. 
\end{prop}

\begin{proof}
Because the Bergman metric is invariant under biholomorphic mappings we may assume $z_{0}=(x,0)$ with   $x=r+\epsilon$,
$\e>0$.
Let $b_{j}=|a_{j0}|^{2}$ and $b_{j}'=|a_{j0}'|^{2}$.  The coefficients  $a_{jk}$s are  known explicitly, see \cite[p. 172]{Pflug}.   
 From \eqref{eq411} with $k=0$    we have
\begin{equation} \label{BJ} 
b_j=\frac{1}{\pi^2} (j+1)(j+2)  \text{ \ and \ } b_{j}'=\gamma_{j}b_{j}  \text{ \ where \ } \gamma_{j}=\frac{1}{1-r^{4+2j}}. 
\end{equation}

Let $z_{0}=(x,0)$, $x\in (r, 1)$. Then we have

\begin{equation}
K_{\W_r}(z_{0},z_{0})=\sum_{j=0}^{\infty}b_{j}'|z_{1}^{j}|^{2}= \sum_{j=0}^{\infty}b_{j}'x^{2j}.\label{eq:Kw}
\end{equation}

For notational convenience we use $\Omega= \Omega_r$. 
 Using  \eqref{KernelDerivative} and $N=(1,0)$  one can calculate the Bergman metric on $\Omega$ as follows:
\begin{multline*}
\left(F_B^{\W}(z_{0},N)\right)^{2}=\p\cl\p\log K_{\W}(z,z)_{z_{0}}(N,\cl N)=\frac{\p^{2}\log K_{\W}(z,z)}{\p z_{1}\p\cl z_{1}}\Bigg|_{z=z_{0}}\\
=\frac{1}{K_{\W}(z_{0})^{2}}\br{\sum_{j=0}^{\infty} b'_{j}x^{2j}\sum_{k=1}^{\infty}b'_{k}k^{2}x^{2k-2} -\br{\sum_{k=1}^{\infty}b'_{k}kx^{2k-1}}^{2}}
\end{multline*}

 We can simplify this as 
 \begin{eqnarray*}
 F_B^{\W}(z_{0},N)^{2}  K_{\W}(z_{0})^{2} 
&=&   \sum_{j=0}^{\infty}b'_{j}x^{2j} \sum_{k=1}^{\infty}b'_{k}k^{2}x^{2k-2}- \br{\sum_{k=1}^{\infty} b'_{k}kx^{2k-1}}^{2}\\[0.1cm] 
&=&  (b'_0 +  \sum_{j=1}^{\infty}b'_{j}x^{2j}) ( \sum_{k=1}^{\infty}b'_{k}k^{2}x^{2k-2}) 
 - \sum_{j,k\geq1}^{\infty}b'_j  b'_{k}jkx^{2(k+j)-2}\\[0.1cm] 
 &=&  b'_0  \sum_{k=1}^{\infty}b'_{k}k^{2}x^{2k-2} +  \sum_{j,k\geq 1}^{\infty}b'_{j}b'_k k^2 x^{2(k+j)-2} 
  - \sum_{j,k\geq1}^{\infty}b'_j  b'_{k}jkx^{2(k+j)-2}. 
\end{eqnarray*} 
Rewriting the second and third terms we have 
\begin{eqnarray*} 
 \sum_{j,k\geq 1}^{\infty}b'_{j}b'_k k^2 x^{2(k+j)-2}
   &=& \sum_{j>k\geq 1}^{\infty}b'_{j}b'_k (k^2+j^2)  x^{2(k+j)-2} + \sum_{k=1}^\infty (b'_k)^2 k^2 x^{4k-2}\\
\sum_{j,k\geq 1}^{\infty}b'_j  b'_{k}  jkx^{2(k+j)-2} 
 &=& 2  \sum_{j>k\geq1}^{\infty}b'_j  b'_{k}  jkx^{2(k+j)-2}     + \sum_{k=1}^\infty (b'_k)^2 k^2 x^{4k-2}. 
\end{eqnarray*} 
Putting together we have 
 \begin{eqnarray*}
 F_B^{\W}(z_{0},N)^{2}  K_{\W}(z_{0})^{2}  
 &=& b'_0  \sum_{k=1}^{\infty}b'_{k}k^{2}x^{2k-2} +  \sum_{j>k\geq 1}^{\infty}b'_{j}b'_k (k-j)^2 x^{2(k+j)-2}\\ 
 &=& \frac12 \sum_{j, k \geq 0}^\infty  b'_{j}b'_k (k-j)^2 x^{2(k+j)-2}. 
 \end{eqnarray*} 

We use the formula  \eqref{eq123} for Bergman metric on the unit ball provided in section~2 to have  $(F_B^{\B}(z_{0},N))^2=3/(1-x^{2})^{2}$. 
Since $K_{\Omega}(z_0,z_0) = \sum b'_j x^{2j}$ and $b_j' = \gamma_jb_j$   our claim of the proposition is 

\begin{equation} \label{eq:Fw1}
(F_B^{\W}(z_{0},N))^2= \frac{\sum_{j,k\ge0}\gamma_{k}\gamma_{j}b_{k}b_{j}x^{2(j+k-1)}(k-j)^{2}} {2\br{\sum_{j=0}^{\infty}\gamma_{j}b_{j}x^{2j}}^{2}}<\frac{3}{(1-x^{2})^{2}}. 
\end{equation}

We rewrite  (\ref{eq:Fw1})  as 

\begin{equation}
(1-2x^{2}+x^{4})\sum_{j,k\geq0}\gamma_{k} \gamma_{j}b_{k}b_{j}x^{2(j+k-1)}(k-j)^{2}-6 \biggl(\sum_{j=0}^{\infty}\gamma_{j}b_{j}x^{2j}\biggr)^{2} =\sum_{l\ge0}C_{l}x^{2l}<0
\end{equation}

We shall show that for $x$ close to $r$ and sufficiently
small  $r>0$, 
\begin{eqnarray*}
 &  & C_{0}<0\text{ and }C_{0}=O(r^{4}),\\
 &  & C_{1}<0\text{ and }C_{1}x^{2}=O(r^{8}),\\
 &  & C_{2}x^{4}=O(r^{8}).
\end{eqnarray*}
 For $l\geq3$, $C_{l}x^{2l}=O(r^{2l})=O(r^{6})$ for $l\geq3$. Hence,
the constant and $x^2$ terms dominate the higher order terms
and they are negative. Hence the claim follows. The rest of the proof
is explicit computation of coefficients.  We recall the formulas for $ b_j $ and $ \gamma_j $ given in  \eqref{BJ}: 
$b_j=(j+1)(j+2)/\pi^2,  b_{j}'=\gamma_{j}b_{j} $ and $ \gamma_{j}=1/(1-r^{4+2j})$. 

The constant term is:

\begin{eqnarray*}
C_{0} & = & 2\gamma_{0}b_{0}\gamma_{1}b_{1}-6\gamma_{0}^{2}b_{0}^{2}\\[0.2cm]
 & = & 2\gamma_{0}b_{0}\frac{1}{\pi^{2}} \biggl(\frac{6}{1-r^{6}}-\frac{6}{1-r^{4}}\biggr)\\[0.2cm]
 & = & -\frac{24}{\pi^{4}}\frac{r^{4}-r^{6}}{(1-r^{6})(1-r^{4})^2}.\end{eqnarray*}

The coefficient of the $x^2$ term is:

\begin{eqnarray*}
C_{1} & = & 2\gamma_{2}\gamma_{0}b_{2}b_{0}2^{2}-2(2\gamma_{1} \gamma_{0}b_{1}b_{0})-6\cdot2\gamma_{1}\gamma_{0}b_{1}b_{0}\\[0.2cm]
 & = & 8\gamma_{2}b_{2}\gamma_{0}b_{0}-16\gamma_{1} \gamma_{0}b_{1}b_{0}\\[0.2cm]
 & = & 8\gamma_{0}b_{0}(\gamma_{2}b_{2}-2\gamma_{1}b_{1})\\[0.2cm]
 & = & -\frac{192}{\pi^{4}}\frac{r^{6}(1-r^{2})}{(1-r^{4}) (1-r^{8})(1-r^{6})}.
\end{eqnarray*}

The coefficient of the $x^4$ term is:

 \begin{eqnarray*}
C_{2} & = & \sum_{j,k\geq0}^{j+k-1=2}\gamma_{k}\gamma_{j} b_{k}b_{j}x^{2(j+k-1)}(k-j)^{2}-2\sum_{j,k\geq0}^{j+k=2} \gamma_{k}\gamma_{j}b_{k}b_{j}x^{2(j+k)}(k-j)^{2}\\[0.2cm]
 &  & +\sum_{j,k\geq0}^{j+k+1=2}\gamma_{k}\gamma_{j} b_{k}b_{j}x^{2(j+k+1)}(k-j)^{2}-6\sum_{j,k\geq0}^{j+k=2} \gamma_{k}\gamma_{j}b_{k}b_{j}x^{2(j+k)}\\[0.2cm]
 & = & \frac{24r^{4}(3-2r^{2}-37r^{4}-66r^{6}-96r^{8}-69r^{10}-34r^{12}+r^{14})} {(1+r^{2})(1+r^{2}+r^{4}+r^{6}+r^{8})(-1+r^{6})(1+r^{2}+r^{4})(-1+r^{8})}.
\end{eqnarray*}

\end{proof}

\medskip
\noindent{\bf Proof of Theorem 3}\\
We may assume $z_0=(x,0)$, $x\in\R$ since the Bergman metric is invariant under biholomorphic mappings. Let $\xi=(\xi_1,\xi_2)=\xi_1N+\xi_2 T$. Then 

\[
(F_{B}^{\W}(z_{0},\xi))^{2}=\sum_{j,k=1}^{2}\frac{\p^{2}\log K_{\W}(z,z)}{\p z_{j}\p\cl z_{k}}\xi_{j}\cl\xi_{k}.
\]

Using equation \ref{KernelDerivative} and the fact that $ K_\W (z,z) = \sum a_{jk} |z_1|^{2j} |z_2|^{2k} $ we have 

\[
\frac{\p^{2}\log K_{\W}(z,z)}{\p z_{1}\p\cl z_{2}}\Bigg|_{(x,0)}=0.
\]

Hence 

\[ (F_{B}^{\W}(z_{0},\xi))^{2}=|\xi_{1}|^{2}(F_{B}^{\W}(z_{0},N))^{2} +|\xi_{2}|^{2}(F_{B}^{\W}(z_{0},T))^{2}.
\]

This equality also holds for the unit ball. Therefore, by Theorem  \ref{602}
and Proposition \ref{ringnormal}, we have that  $F_{B}^{\W}(z_{0},\xi)<F_{B}^{\B}(z_{0},\xi)$
assuming $r$ is small enough and $z_{0}$ is close to the inner boundary. \qed

\bigskip{}

\noindent Hyunsuk Kang\\
 Mathematics Department\\
 The University of Michigan\\
 East Hall, Ann Arbor, MI 48109\\
 USA\\
 hskang@umich.edu\\

\noindent Lina Lee\\
 Mathematics Department\\
 The University of Michigan\\
 East Hall, Ann Arbor, MI 48109\\
 USA\\
 linalee@umich.edu\\

\noindent Crystal Zeager\\
 Mathematics Department\\
 The University of Michigan\\
 East Hall, Ann Arbor, MI 48109\\
 USA\\
 zeagerc@umich.edu\\

\end{document}